%% file: manuscript3.tex
\documentclass[preprint,12pt]{elsarticle}

\usepackage{amssymb}
\usepackage{amsthm}
\usepackage{amsmath}
\include{environ}
\include{symbols}
\usepackage{jen_macros}
\usepackage{multirow}
\newtheorem{thm}{Theorem}
\usepackage[usenames]{color}

\begin{document}

\input{algmacros}

\begin{frontmatter}



\title{Limited-memory BFGS Systems with Diagonal Updates}

\author[jbe]{Jennifer B. Erway\fnref{label1}\corref{cor1}}
\ead{erwayjb@wfu.edu}
\address[jbe]{Department of Mathematics\\
         Wake Forest University \\
         Winston-Salem, NC 27109}

\author[rfm]{Roummel F. Marcia\fnref{label2}}
\ead{rmarcia@ucmerced.edu}
\address[rfm]{5200 N. Lake Road,\\
         University of California, Merced\\ 
         Merced, CA 95343}

\fntext[label1]{Supported in part by NSF grant DMS-08-11106.}
\fntext[label2]{Supported in part by NSF grant DMS-0965711.}

\cortext[cor1]{Corresponding author, phone number: (336) 758-5356}

\begin{abstract}
  In this paper, we investigate a formula to solve systems of the form
  $(B+\sigma I)x=y$, where $B$ is a limited-memory BFGS quasi-Newton matrix
  and $\sigma$ is a positive constant. These types of systems arise
  naturally in large-scale optimization such as trust-region methods as well as
  doubly-augmented Lagrangian methods.  We show that provided a simple
  condition holds on $B_0$ and $\sigma$, the system $(B+\sigma I)x=y$ can
  be solved via a recursion formula that requies only vector inner
  products.  This formula has complexity $M^2n$, where $M$ is the number of
  L-BFGS updates and $n \gg M$ is the dimension of $x$.
\end{abstract}


\end{frontmatter}

\newcommand{\trace}{\mathop{\mathrm{trace}}}

\newcommand{\tWhite}[1]{{\color{white}#1}}
\newcommand{\defined}{\mathop{\,{\scriptstyle\stackrel{\triangle}{=}}}\,}
\newcommand{\words}[1]{\mgap\textrm{#1}\mgap}
\newcommand{\mgap}{\;\;}
\newcommand{\subject}{\hbox{\rm subject to}}
\newcommand{\minimize}[1]{{\displaystyle\minim_{#1}}}
\newcommand{\minim}{\mathop{\mathrm{minimize}}}
\newcommand{\diag}{\mathop{\mathrm{diag}}}
\newcommand{\strict}{\mathop{\mathrm{strict}}}
\newcommand{\interior}{\mathop{\mathrm{int}}}
\newcommand{\st}{\mathop{\mathrm{ : }}}
\newcommand{\BigO}[1]{O\big(#1\big)}
\newcommand{\strt}{\rule[-.5ex]{0pt}{3ex}}
\newcommand{\hstrt}{\rule[-1ex]{0pt}{3.5ex}}
\newcommand{\bgap}{\;\;\;}


\section{Introduction}
\label{sec:intro}

In this paper we develop a recursion formula for solving systems of the
form $(B_k+\sigma I)x=y$, where $B_k$ is the $k$-th step $n\times n$
limited-memory (L-BFGS) quasi-Newton Hessian
(see e.g., \cite{ByrNS94,LiuN89,Morales02,Noc80}), $\sigma$ is a positive
constant and $x,y \in \Re^n$.  Linear systems of this form appear in both
large-scale unconstrained and constrained optimization.  For example, these
equations arise in the optimality conditions for the two-norm trust-region
subproblem and the so-called \emph{doubly-augmented} system and their
applications~\cite{ErwG09,ErwGG09,ForGG07}.  Additionally, matrices of the
form $B_k+\sigma I$ can be used as preconditioners for $H+\sigma I$, where
$H$ is often the Hessian of a twice-continuously differentiable function
$f$.

The algorithm proposed in this paper is an extension of the
Sherman-Morrison-Woodbury (SMW) formula to compute the inverse of
$B_k+\sigma I$.  The algorithm begins by pairing the initial L-BFGS matrix
$B_0$ with the initial diagonal update $\sigma I$; then, the algorithm uses
the SMW formula to compute the inverse of $B_k+\sigma I$ after updating
with a sequence of the L-BFGS updates.  In this paper, we derive a
recursion formula for efficiently computing matrix-vector products with
this inverse.  The proposed algorithm requires $\mathcal{O}(M^2n)$
multiplications, where $M$ is the maximum number of L-BFGS updates.

The structure of the remainder of the paper is as follows.  In
\S\ref{sec:applications}, we motivate in detail why solving systems of the
form $(B_k+\sigma I)x=y$ is crucial in several optimization settings.
Specifically, we consider its use in solving the trust-region subproblem
and in the preconditioning of doubly-augmented system in barrier methods.
In \S\ref{sec:lbfgs}, we provide an overview of the L-BFGS quasi-Newton
matrix, including operation counts of the well-known recursive formulas
for operations with the quasi-Newton matrix.  In \S\ref{sec:formula}, we
consider the formula for solving systems of the form $(B_k+\sigma I)x=y$
and show how it compares computationally with direct and indirect methods
in \S\ref{sec-experiments}.   We offer potential extensions of this formula
in \S\ref{sec-extensions} and draw some conclusions in \S\ref{sec:conc}.

In this paper, we assume that updates are chosen that ensure all L-BFGS
quasi-Newton matrices are positive definite.

\section{Motivations from large-scale optimization}
\label{sec:applications}
Systems of the form $(B+\sigma I)x=y$, where $B$ is a quasi-Newton Hessian
appear throughout large-scale, nonlinear optimization.  In this section, we
motivate our research by presenting two instances in optimization that
would benefit from a recursion formula to directly solve systems with a
system matrix of the form $B+\sigma I$.  The first motivation is in
trust-region methods for unconstrained optimization; the second motivation
comes from barrier methods for constrained optimization.

\bigskip

\noindent \textbf{Motivation 1: Trust-region methods.}  Trust region
methods are one of the most popular types of methods for unconstrained
optimization.  Trust region methods have been extended into the
quasi-Newton setting by using BFGS or L-BFGS updates (see, for
example,~\cite{BurWX96, David76, DenM79, Erwa11, Ger04, Kauf99, NiY97, 
  Pow70d, Pow70c, YuaW06}).  The bulk of the work for a trust-region method
occurs when solving the trust-region subproblem.  Given the current
trust-region iterate $\bar{x}$, the two-norm trust-region subproblem for
minimizing a function $f$ is given by
\begin{equation} \label{eqn-trustProblem}
     \minimize{s\in\Re^n}\mgap q(s)\defined g^T s + \frac{1}{2} s^T B  s
     \bgap\subject            \mgap  \|s\|_2 \le \delta,
\end{equation}
where $g\defined \nabla f(\bar{x})$, $B$ is an L-BFGS quasi-Newton
approximation to the Hessian of $f$ at $\bar{x}$, and $\delta$ is a
given positive trust-region radius.

Optimality conditions for the L-BFGS quasi-Newton trust-region subproblem are
summarized in the following result 
(adapted from \cite{ConGT00a,Gay81,MorS84,Sor82}).
\begin{thm}\label{thmTrustProblem}
  Let $\delta$ be a given positive constant. A vector $s^*$ is a
  global solution of the trust-region problem
  {\rm(\ref{eqn-trustProblem})} if and only if $\|s^*\|_2 \le
  \delta$ and there exists a unique $\sigma^*\ge0$ such that
\begin{equation}                              \label{eqnUC-TR-optimality}
  (B+\sigma^* I)s^* = - g  \words{and}
   \sigma^*(\delta - \|s^*\|_2)=0.
\end{equation}
Moreover, the global minimizer is unique.
\end{thm}
\quad\endproof

The Mor\'e-Sorensen method~\cite{MorS83} is the preferred \emph{direct}
solver for the general trust-region subproblem.  The method seeks 
a solution $(x^*,\sigma^*)$ that satisfies the optimality conditions for
the trust-region subproblem (in this case (\ref{eqnUC-TR-optimality})) to
any prescribed accuracy.  The algorithm below summarizes the
Mor\'e-Sorensen method~\cite{MorS83} for the \emph{general} trust-region
subproblem:

\begin{Pseudocode}{Algorithm 2.1: Mor\'e-Sorensen method}\label{alg-ms}
\tab Let $\sigma\ge 0$  with $B+\sigma I$ positive definite and $\delta>0$
\tab \WHILE{} $\NOT \converged$ \DO\
\tabb Factor $B+\sigma I=R^TR$; 
\tabb Solve $R^TRp=-g$;
\tabb Solve $R^Tq=p$;
\tabb $\sigma\leftarrow \sigma+\left(\frac{\|p\|_2}{\|q\|_2}\right)^2
\left(\frac{\|p\|_2-\delta}{\delta}\right);$
\tab \END\
\end{Pseudocode}

Notice that the Mor\'e-Sorensen method solves systems of the form
$(B+\sigma I)s=-g$ at each iteration by computing the Cholesky
factorization of $B+\sigma I$.  For general large-scale optimization, where
$B$ is not a quasi-Newton Hessian, this method is often prohibitively
expensive if one cannot exploit structure in the system matrix $B$.
However, in the quasi-Newton setting, it is possible to compute the
Choleksy factorization of a BFGS matrix of $B_{k+1}$ by updating the
factorization for $B_k$ (see, for example,~\cite{GilGMS07}).

In this paper, we develop a method for solving the system $(B+\sigma
I)s=-g$ using a recursion relation in place of using Cholesky
factorizations. It is then possible to continue on with the Mor\'e-Sorensen
method by updating $\sigma$ in accordance with the ideas proposed
in~\cite{MorS83}.  (Note: The source of the formula for updating $\sigma$
in Algorithm~\ref{alg-ms} is based on applying Newtons method to the second
optimality condition in (\ref{eqnUC-TR-optimality}).)  The recursion
formula proposed in this paper extends the Mor\'e-Sorensen method into the
quasi-Newton setting without having to update Cholesky factorizations.

\bigskip

\noindent \textbf{Motivation 2: Barrier methods.}  The second motivating
example comes from barrior methods for constrained optimization 
(see, e.g., \cite{Fiacco1990,Wright05}).
Consider the following problem:
\begin{equation} \label{eqn-2nd-application}
     \minimize{x\in\Re^n}\mgap f(x)
\bgap\subject            \mgap  c(x) \ge 0,
\end{equation}
where $f(x):\Re^{n}\rightarrow\Re$ is a real-valued function and
$c(x):\Re^{n}\rightarrow\Re$ is a quadratic constraint of the form
$c(x)=\frac{1}{2}\delta^2-\frac{1}{2}x^Tx$, i.e., a two-norm
constraint on the size of $x$ where $\delta$ is a positive constant.
(Note this can be considered a generalization of the 
trust-region subproblem.)  A sufficient condition for a point $x^*$ to
be the minimizer of (\ref{eqn-2nd-application}) is the existence of
the Lagrange multiplier $\lambda^*$ satisfying the following:
\begin{equation}                      \label{eqn-optimality-conditions}
\begin{array}{r@{\hspace{4pt}}c@{\hspace{4pt}}l%
                 @{\hspace{10pt}}l}
\nabla c(x^*)\lambda^*&=&g(x^*) ,  &\words{with} H(x^*)+ \lambda^*I \words{positive definite,} \\
      c(x^*)\lambda^* &=&\, 0,  &\words{with} \lambda^*>0 \words{if} c(x^*)=0 \words{and} \lambda^*, c(x^*) \ge 0,
\end{array}
\end{equation}
where $g(x) = \nabla f(x)$ is the gradient and $H(x) = \nabla^2 f(x)$ is the Hessian of $f(x)$.  
Primal-dual methods  \cite{DenS96, ForGW02, NocW06} solve this type of problem
by solving a sequence of \emph{perturbed} problems.  Specifically, we define
the function $F_{\mu} \colon \Re^{n+1} \rightarrow \Re^{n+1}$ with
$$
	 F_{\mu}(x,\lambda)=
	\begin{pmatrix}
		 g(x)+\lambda x \\
		 c(x)\lambda -\mu
	\end{pmatrix},
$$
for some fixed perturbation parameter $\mu > 0$.
Note that $\nabla c(x) = -x$ and as $\mu \rightarrow 0$, 
the root of $F_\mu(x,\lambda)$ converges to a point that satisfies 
the equations in (\ref{eqn-optimality-conditions}).

The Newton equations associated with finding a root of $F_\mu(x,\lambda)$ are
given by
\begin{equation}\label{eqn-newtoneqns-app2}
\begin{pmatrix}H(x)+\lambda I & x \\ -\lambda x^T  & c(x) \end{pmatrix}
\begin{pmatrix}\varDelta x \\ \varDelta \lambda \end{pmatrix}
= -
\begin{pmatrix}g(x)+ \lambda x \\ c(x)\lambda-\mu\end{pmatrix}.
\end{equation}
By dividing the second row by $-\lambda$ and letting $d \defined c(x)/\lambda$, 
we get the symmetric system
\begin{equation}\label{eqn-newNewton}
 \begin{pmatrix}H(x)+\lambda I &  x \\ x^T & -d
 \end{pmatrix}
\begin{pmatrix}\varDelta x \\ \varDelta\lambda \end{pmatrix}
= -
\begin{pmatrix}g(x)+\lambda x \\ (\mu-d)/\lambda \end{pmatrix}.
\end{equation}
Unfortunately, the system matrix in (\ref{eqn-newNewton}) is indefinite
even when $H+\lambda I$ is positive definite; however, after rearranging terms, it can be shown that this system is equivalent to the
\emph{doubly-augmented} system~\cite{ForGG07}:
\begin{equation}\label{eqn-doubly-aug2}
 \begin{pmatrix}H(x)+\lambda I+2xx^T &  -x \\ -x^T & d
 \end{pmatrix}
\begin{pmatrix}\varDelta x \\ \varDelta\lambda \end{pmatrix}
= -
\begin{pmatrix}g(x)+ \lambda x+(2/d)(\mu-d)x \\ (d-\mu)/\lambda\end{pmatrix},
\end{equation}
which is a positive-definite system when $H+\lambda I$
is positive definite~\cite{ForGG07}.

The linear system (\ref{eqn-doubly-aug2}) can be preconditioned by the matrix
\begin{equation}\label{eqn-matP}
P=\begin{pmatrix}
B+\lambda I + 2xx^T &   -x \\
-x^T & d
\end{pmatrix},
\end{equation}
where $B$ is an L-BFGS quasi-Newton approximation to $\nabla^2 f(x)$.
Solves with the preconditioner $P$ for a general system of the form
$$
     \begin{pmatrix}  B+\lambda I + 2xx^T & - x\\
                      -x^T         &\, d      \end{pmatrix}
     \begin{pmatrix} x_1 \\
            x_2   \end{pmatrix}
 \ = \ \begin{pmatrix} y_1 \\
            y_2   \end{pmatrix}    
$$
can be performed by solving the following equivalent system:
\begin{equation}\label{eqn-equiv}
 \begin{pmatrix}  B+\lambda I & \,x \\
             x^T &  - d   \end{pmatrix}
 \begin{pmatrix} x_1 \\
        x_2   \end{pmatrix}
 \ = \ \begin{pmatrix}  y_1 + 2y_2x\\
            -y_2              \end{pmatrix} 
\end{equation}
(see ~\cite{ErwG09,ForGG07}).
Note that the inverse of the system matrix in (\ref{eqn-equiv}) is given by
$$
 \begin{pmatrix}\,  B+\lambda I & \,\bar{x}\\
          \bar{x}^T    &  - 1       \end{pmatrix}^{-1}
 = 
   \begin{pmatrix} I &   -w  \\
          0 &   1  \end{pmatrix}
   \begin{pmatrix}  (B+\lambda I)^{-1} &         0             \\
                0     & -\big(1+w^T \bar{x}\big)^{-1}  \end{pmatrix}
\begin{pmatrix}  I   &  0  \\
           -w^T &  1    \end{pmatrix},
$$
where $w=(B+\lambda I)^{-1}x$ (for details, see ~\cite{ErwG09,ForGG07}). 
Thus, provided that solves with $B+\lambda I$ are efficient, solves with $P$ 
are efficient.  

In this paper, we develop a recursion formula that makes solves with
$(B+\lambda I)x=y$.  Thus, this recursion formula allows the use of
preconditioners the form (\ref{eqn-matP}) where $B$ is a L-BFGS
quasi-Newton approximation of $\nabla^2 f$.

\section{The Limited-memory BFGS method}
\label{sec:lbfgs}

In this section, we review the limited-memory BFGS (L-BFGS) method and state
important and well-known recursion formulas for computing products with an
L-BFGS quasi-Newton Hessian and its inverse.

The BFGS quasi-Newton method generates a sequence of positive-definite
matrices $\{B_k\}$ from a sequence of vectors
$\{y_k\}$ and $\{s_k\}$ defined as
$$y_k=\nabla f(x_{k+1})-f(x_k) \quad  \textrm{and} \quad s_k=x_{x+1}-x_k,$$
respectively.  The L-BFGS quasi-Newton method can be
viewed as the BFGS quasi-Newton method where only at most $M$ recently
computed updates are stored and used to update the initial matrix $B_0$.
Here $M$ is a positive constant with $M\ll n$.
The L-BFGS quasi-Newton approximation to the Hessian of $f$ is implicitly
updated as follows:
\begin{equation}
B_{k} = B_0 - \sum_{i=0}^{k-1} a_i {a_i}^T + \sum_{i=0}^{k-1} b_ib_i^T,
\label{eqn-bfgs}
\end{equation}
where 
\begin{equation}
a_i=\frac{B_is_i}{\sqrt{s_i^TB_is_i}}, 
\quad b_i = \frac{y_i}{\sqrt{y_i^T s_i}}, \quad 
B_0=\gamma_k^{-1} I, \label{eqn-bfgs-extra}
\end{equation}
and $\gamma_k>0$ is a constant.  In practice, $\gamma_k$ is often taken to
be $\gamma_k\defined s_{k-1}^Ty_{k-1}/\|y_{k-1}\|_2^2$ (see,
e.g.,~\cite{LiuN89} or~\cite{Noc80}).  

Suppose that we have computed $k$ updates ($k\leq M-1$) and have the
following updates stored in $S$ and $Y$:
$$ S= [s_0  \ldots s_{k-1}] \quad \textrm{and}
\quad Y=[y_0 \ldots y_{k-1}].$$ We update $S$ and $Y$ with
the most recently computed vector pair $(s_{k}, y_{k})$
as follows: 
\newpage
\begin{Pseudocode}{Algorithm 3.1: Update $S$ and $Y$}\label{alg-update}
\tab \IF\ $k<M-1$,
\tabb $S\leftarrow [S \,\,\, s_k]; \mgap Y\leftarrow [Y \,\,\, y_k]; \mgap
k\leftarrow k+1$;
\tab \ELSE\ 
\tabb \FOR\ $i=0,\ldots k-1$
\tabbb  $s_i\leftarrow s_{i+1}; \mgap y_i\leftarrow y_{i+1}$; 
\tabb \END\
\tabb $S\leftarrow [s_0,\dots s_{k-1}]; \mgap Y\leftarrow [y_0,\dots y_{k-1}]$; 
\tab \END\
\end{Pseudocode}

\noindent Thus, at all times we have exactly
$k$ stored vectors with $k\leq M-1$.

\medskip


For the L-BFGS method, there is an efficient recursion relation to compute
products with $B_k^{-1}$.  Given a vector $z$, 
the following algorithm~\cite{Noc80,NocW06} terminates with 
$r\defined B_k^{-1}z$:
\newline
\newline
\begin{Pseudocode}{Algorithm 3.2: Two-loop recursion to compute $r=B_k^{-1}z$}\label{alg-recursion}
  \tab $q\leftarrow z$;
  \tab \FOR\ $i=k-1,\dots,0$
  \tabb $\alpha_i\leftarrow(s_i^Tq)/(y_i^Ts_i)$;
  \tabb $q \leftarrow q-\alpha_iy_i$;
  \tab \END\ 
  \tab $r\leftarrow B_0^{-1} q$;
  \tab \FOR\ $i=0,\ldots, k-1$
  \tabb  $\beta \leftarrow (y_i^Tr)/(y_i^Ts_i)$;
  \tabb  $r \leftarrow r+(\alpha_i-\beta)s_i$:
  \tab \END\
\end{Pseudocode}

Pre-computing and storing $1/y_i^Ts_i$ for $0\leq i\leq k-1$, makes
Algorithm~\ref{alg-recursion} even more efficient.  Further details on the
L-BFGS method can found in~\cite{NocW06}; further background on the BFGS
can be found in~\cite{DenS96}.  The two-term recursion formula requires at
most $\mathcal{O}(Mn)$ multiplications and additions.  There is compact
matrix representation of the L-BFGS that can be used to compute products
with the L-BFGS quasi-Newton matrix (see, e.g., ~\cite{NocW06}). The
computational complexity at most $\mathcal{O}(Mn)$ multiplications.

\medskip 

There is an alternative representation of $B_k^{-1}$ from which the two-term
recursion can be more easily understood:
\begin{eqnarray}
  B^{-1}_{k} =  (V_{k-1}^T\cdots V_{0}^T ) B_0^{-1} (V_{0}\cdots V_{k-1}) &+&
  \frac{1}{y_{0}^Ts_{0}} (V_{k-1}^T\cdots V_{1}^T)s_0s_0^T(V_1\cdots V_{k-1}) \nonumber\\
  & + &
  \frac{1}{y_{1}^Ts_{1}} (V_{k-1}^T\cdots V_{2}^T)s_1s_1^T(V_2\cdots V_{k-1})\nonumber \\
  & + & \cdots \nonumber \\ 
  & + & \frac{1}{y_{k-1}^Ts_{k-1}} s_{k-1}s_{k-1}^T, \label{eqn-lbfgs2}
\end{eqnarray}
where $V_i=I-\frac{1}{y_i^Ts_i}y_is_i^T$ (see, e.g.,~\cite{NocW06}).  The
first loop in the two-term recursion for $B_k^{-1}z$ computes and stores
the products $(V_j\cdots V_{k-1})z$ for $j=0,\ldots, k-1$; in between the
first and second loop, $B_0^{-1}$ is applied; and finally, the second loop
computes the remainder of (\ref{eqn-lbfgs2}).  Computing the inverse of
$B_k+\sigma I$ is not equivalent to simply replacing $B_0^{-1}$ in the
two-loop recursion with $(B_0+\sigma I)^{-1}$.  To see this, notice that
replacing $B_0^{-1}$ in (\ref{eqn-lbfgs2}) with $(B_0+\sigma I)^{-1}$ would
apply the updates $V_i$ to the full quantity $(B_0+\sigma I)^{-1}$ instead
of only $B_0^{-1}$.  The main contribution of this paper is a recursion
formula that computes $(B_k+\sigma I)^{-1}z$ in an efficient manner using only vector inner products.
\section{The recursion formula}
\label{sec:formula}
Consider the problem of finding the inverse of $B+\sigma I$, where $B$ is
an L-BFGS quasi-Newton matrix.  The Sherman-Morrison-Woodbury (SMW) formula
gives the following formula for computing the inverse of $A+UV^T$, assuming
$A$ is invertible (see~\cite{GolV96}):
$$
(A+UV^T)^{-1}=A^{-1}-A^{-1}U(I+V^TA^{-1}U)^{-1}V^TA^{-1}.
$$
In the special case when $UV^T$ is a rank-one update to A, this formula becomes
$$(A+uv^T)^{-1}=A^{-1}-A^{-1}u(I+v^TA^{-1}u)^{-1}v^TA^{-1},$$
where $u$ and $v$ are both $n$-vectors.  For simplicity, first consider
computing the inverse of an L-BFGS quasi-Newton matrix after only one
update, i.e., the inverse of $B_1+\sigma I$.  Recall that
$$
B_1+\sigma I = (\gamma_1^{-1}+\sigma) I -a_0a_0^T+b_0b_0^T.
$$
To compute the inverse of this, we apply SMW twice.  To see this clearly,
let $$C_0=(\gamma_1^{-1}+\sigma) I, \quad C_1= (\gamma_1^{-1}+\sigma) I -a_0a_0^T, 
\quad C_2= (\gamma_1^{-1}+\sigma) I -a_0a_0^T+b_0b_0^T.
$$
Applying SMW to $C_1$ yields
\begin{eqnarray}
C_1^{-1} & = & C_0^{-1}+C_0^{-1}a_0(1-a_0^TC_0^{-1}a_0)^{-1}a_0^TC_0^{-1}\\
         & = & C_0^{-1}+\frac{1}{(1-a_0^TC_0^{-1}a_0)}C_0^{-1}a_0a_0^TC_0^{-1}.
\end{eqnarray}
Applying SMW once more we obtain $C_2^{-1}$ from $C_1^{-1}$:
$$ C_2^{-1}= C_1^{-1}-\frac{1}{(1+b_0^TC_1^{-1}b_0)}C_1^{-1}b_0b_0^TC_1^{-1},$$
giving an expression for $(B_1+\sigma I)^{-1}$.  This is the basis for the
following recursion method that appears in~\cite{Mill81}:

\begin{thm}\label{thrm-maa}
  Let $G$ and $G+H$ be nonsingular matrices and let $H$ have positive rank
  $M$.  Let $H=E_0+E_1+\cdots + E_{M-1}$ where each $E_k$ has rank one and
  $C_{k+1}=G+E_0+\cdots +E_k$ is nonsingular for $i=0,\ldots M-1$.  Then if
  $C_0=G$, 
\begin{equation}\label{eqn-maa-recursion}  
C_{k+1}^{-1}=C_k^{-1}-v_kC_k^{-1}E_kC_{k}^{-1}, \mgap k=0,\ldots, M-1,
\end{equation}
  where 
\begin{equation}
	 v_k = \frac{1}{1+\trace\left(C_k^{-1}E_k\right)}.	\label{eqn-vk}
\end{equation}
    In particular,
\begin{equation}\label{eqn-g+h}
(G+H)^{-1}=C_{M-1}^{-1}-v_{M-1}C_{M-1}^{-1}E_{M-1}C_{M-1}^{-1}.
\end{equation}
\end{thm}
\begin{proof}
See~\cite{Mill81}.  
\end{proof}

\medskip

We now show that applying the above recursion method to $B_k+\sigma I$, 
the product $(B_k+\sigma I)^{-1}z$ can be computed recursively, assuming
$\gamma_k\sigma$ is bounded away from zero.
\begin{thm}\label{thrm-maa2} Let $\gamma_k>0$ and $\sigma > 0$ with $\gamma_k \sigma > \epsilon$
for some $\epsilon > 0$.  Let $G= B_0 + \sigma I = (\gamma_k^{-1}+\sigma )I$, and let 
$H = \sum_{i = 0}^{2k-1} E_i$, where 
$$
  E_0=-a_0a_0^T, \mgap
  E_1=b_0b_0^T, \mgap \ldots, \mgap E_{2k-2}=-a_{k-1}a_{k-1}^T, \mgap
  E_{2k-1}=b_{k-1}b_{k-1}^T.
$$ 
Then $(B_k+\sigma I)^{-1} = (G+H)^{-1}$
is given by (\ref{eqn-g+h})  together with (\ref{eqn-maa-recursion}).
\end{thm}

\begin{proof}
  Notice that this theorem follows from 
  Theorem~\ref{thrm-maa}, provided we satisfy its assumptions.
  By construction, $B_k+\sigma I = G+H$. Both $B_k$ and $B_k+\sigma
  I$ are nonsingular since $B_k$ is positive definite and $\sigma \ge
  0$.  It remains to show that $C_{j}$, which is given by
$$
	C_j =  G + \sum_{i=0}^{j-1}E_i = (B_0 +  \sum_{i=0}^{j-1}E_i) + \sigma I,
$$ 
is nonsingular for $j=1,\ldots k-1$, for which we use induction on $j$.  

  Since $C_1=C_0-a_0a_0^T = C_0 ( I - C_0^{-1}a_0 a_0^T)$, 
  the determinant of $C_1$ and $C_0$ are related as follows \cite{DenM77}:
  $$
  	\det(C_1) = \det(C_0)(1-a_0^TC_0^{-1}a_0).
  $$
  In other words, $C_1$ is invertible if $C_0$ is invertible and
  $a_0C_0^{-1}a_0\neq 1$.  We already established that $C_0$ is
  invertible; to show the latter condition, we use the definition of
  $a_0=B_0s_0/\sqrt{s_0^TB_0s_0}$ together with
  $C_0^{-1}=(\gamma_k^{-1}+\sigma)^{-1} I$ to obtain the following:
\begin{eqnarray} 
  a_0^TC_0^{-1}a_0 
&  = &
  \frac{\gamma_k^{-2}(\gamma_k^{-1}  +\sigma)^{-1}}{\gamma_k^{-1}s_0^Ts_0} s_0^Ts_0 \nonumber \\
 & = &
  \frac{1}{\gamma_k(\gamma_k^{-1}+\sigma)} \nonumber \\
&  = &
\frac{1}{1+\gamma_k\sigma}. \label{eqn-aC0a}
\end{eqnarray}
By hypothesis, $\gamma_k\sigma>\epsilon$, which implies that det$(C_1) \ne 0$; 
thus, $C_1$ is invertible.  

Now we assume that $C_j$ is invertible and show that $C_{j+1}$ is
invertible.  If $j$ is odd, then $j + 1 = 2i$ for some $i$ and $C_{j+1} =
B_i + \sigma I$, which is positive definite and therefore nonsingular.  If
$j$ is even, i.e., $j= 2i$ for some $i$, then $C_j = B_i + \sigma I,$ and
\begin{eqnarray*}
	C_{j+1} \ = \ 
	C_j - a_{i}a_{i}^T \ = \
	 B_{i} - \frac{1}{s_{i}^T B_{i} s_{i}} B_{i} s_{i} s_{i}^T B_{i}^T + \sigma I 
\end{eqnarray*}
We will demonstrate that $C_{j+1}$ is nonsingular by showing that it is positive definite.  Consider $z \in \mathbb{R}^n$ with $z \ne 0$.  Then 
\begin{eqnarray}
	z^T C_{j+1} z &=& z^T \! \left ( B_{i} - \frac{1}{s_{i}^T B_{i} s_{i}} B_{i} s_{i} s_{i}^T B_{i}^T \right ) \! z
	+ \sigma \| z \|_2^2 \nonumber \\
	&=& z^T B_{i} z - \frac{(z^T B_i s_i)^2}{s_{i}^T B_{i} s_{i}} 
	+ \sigma \| z \|_2^2  \nonumber \\
	&=& \| B_i^{1/2} z \|_2^2 -
	 \frac{\big ((B_i^{1/2}z)^T(B_i^{1/2}s_i) \big )^2}{ \| B_i^{1/2} s_i \|_2^2}
	+ \sigma \| z \|_2^2  \nonumber
	\\
	&=& \| B_i^{1/2} z \|_2^2 -  \| B_i^{1/2} z \|_2^2  \cos^2 \! \big ( 
	\angle ( B_i^{1/2}z, B_i^{1/2} \! s_i ) \big )  + \sigma \| z \|_2^2 \nonumber \\[.2cm]
	&>& 0. \label{eqn-zCkz}
\end{eqnarray}
We have now satisfied all the assumptions of Theorem \ref{thrm-maa}.  Therefore,
$(B_j+\sigma I)^{-1}$ is given by (\ref{eqn-g+h})  together with (\ref{eqn-maa-recursion}).
\end{proof}

Now, we show $r = C_{k+1}^{-1}z$ in (\ref{eqn-maa-recursion}) can be
computed efficiently using recursion.  We note that using
(\ref{eqn-maa-recursion}), we have
$$
	C_{k+1}^{-1}z \ = \
	C_k^{-1}z - v_k C_k^{-1} E_k C_k^{-1}z
	\ = \
	\begin{cases}
		C_k^{-1}z + v_k C_k^{-1} a_{\frac k2} a_{\frac k2}^T C_k^{-1} z
		 & \textrm{if $k$ is even} \\
		C_k^{-1}z - v_k C_k^{-1} b_{\frac{k-1}{2}} b_{\frac{k-1}{2}}^TC_k^{-1} z &\textrm{if $k$ is odd.}
	\end{cases}
$$
The quantity $v_k$ is obtained using (\ref{eqn-vk}) and computing
$\trace(C_k^{-1}E_k)$, which after substituting in the definition of $E_k$
and computing the trace, is given by
$$
\trace(C_k^{-1}E_k)=
\begin{cases}
-a_{k/2}^TC_k^{-1}a_{k/2} & \textrm{if  $k$ is even}\\
\ \ \ \ \! b_{(k-1)/2}^TC_k^{-1}b_{(k-1)/2} & \textrm{if $k$ is odd}
\end{cases}
$$
If we define $p_k$ according to the following rules
\begin{equation} \label{eqn-pk}
	p_k \ = \ 
	\begin{cases}
		C_k^{-1}a_{k/2} & \textrm{if $k$ is even} \\
		C_k^{-1}b_{(k-1)/2} & \textrm{if $k$ is odd,}
	\end{cases}
\end{equation}
then
\begin{eqnarray*}
	v_k &=& \begin{cases} 
	\displaystyle \frac{1}{1 - p_k^Ta_{k/2}} & \textrm{if $k$ is even} \\
	\displaystyle \frac{1}{1 + p_k^Tb_{(k-1)/2}} & \textrm{if $k$ is odd,}
	\end{cases}
\end{eqnarray*}
and thus,
$$
	C_{k+1}^{-1}z \ = \
	C_k^{-1}z + (-1)^k v_k (p_k^T z) p_k.
$$
Applying this recursively yields the following formula:
\begin{equation} \label{eqn-Ck1}
	C_{k+1}^{-1}z \ = \ C_0^{-1}z + \sum_{i=0}^k (-1)^i v_i (p_i^Tz) p_i,
\end{equation}
for $k\ge 0$ and with $C_0^{-1}z = (\gamma_k^{-1} + \sigma)^{-1}z$.

Finally, it remains to demonstrate how to compute $p_k$ in (\ref{eqn-pk}).
Notice that we can compute $p_k$ using (\ref{eqn-Ck1}):
$$
	p_k = 
	\begin{cases}
		\displaystyle
		C_k^{-1}a_{k/2}  \hspace{.55cm} \ = \ 
		C_0^{-1}a_{k/2} + \sum_{i = 0}^{k-1}(-1)^iv_i(p_i^Ta_{k/2})p_i & \textrm{if $k$ is even} \\
		\displaystyle
		C_k^{-1}b_{(k-1)/2} \ = \ 
		C_0^{-1}b_{(k-1)/2} + \sum_{i = 0}^{k-1}(-1)^iv_i(p_i^Tb_{(k-1)/2})p_i
		& \textrm{if $k$ is odd.}
	\end{cases}
$$
The following pseudocode summarizes the algorithm for computing $r =
C_{k+1}^{-1}z$:

\begin{Pseudocode}{Algorithm 4.1: Proposed recursion to compute $r=C_{k+1}^{-1}z$}\label{alg-recursionC}
  \tab $r \leftarrow (\gamma_{k+1}^{-1} + \sigma)^{-1}z$;
  \tab \FOR\ $j = 0, \dots, k$
  \tabb \IF\ $j$ even
  \tabbb $c \leftarrow a_{j/2}$;
  \tabb \ELSE
  \tabbb $c \leftarrow b_{(j-1)/2}$;
  \tabb \END\
  \tabb $p_j \leftarrow (\gamma_{k+1}^{-1}+\sigma)^{-1}c$;
  \tabb \FOR\ $i = 0, \dots, j-1$
  \tabbb $p_j \leftarrow p_j + (-1)^iv_i (p_i^Tc) p_i$;
  \tabb \END\
  \tabb $v_j \leftarrow 1/(1 + (-1)^{j}p_j^Tc)$;
  \tabb $r \leftarrow r + (-1)^{i}v_i (p_i^Tz) p_i$;
  \tab \END\

\end{Pseudocode}

Algorithm~\ref{alg-recursionC} requires $\mathcal{O}(k^2)$ vector inner
products. Operations with $C_0$ and $C_1$ can be hard-coded since $C_0$ is
a scalar-multiple of the identity.  Experience has shown that $k$ may be
kept small (for example, Byrd et al.~\cite{ByrNS94} suggest $k\in [3,7]$),
making the extra storage requirements and computations affordable. 

\section{Numerical Experiments}
\label{sec-experiments}

We demonstrate the effectiveness of the proposed recursion formula by solving 
linear systems of the form (\ref{eqn-bfgs}) with various sizes.  Specifically, we let
the number of updates $M=5$ and the size of the matrix range from $n = 10^3$ up to $10^7$.
We implemented the proposed method in Matlab on a Two 2.4 GHz Quad-Core Intel Xeon 
``Westmere''  Apple Mac Pro and compared it to a direct method using the 
Matlab ``backslash'' command and the built-in conjugate-gradient (CG) method (pcg.m).  
Because of limitations in memory, we were only able to use the direct method for problems where
$n \le 20,000$.  In the tables below, we show the time and the relative residuals for each method.
The relative residuals for the recursion formula are used as the criteria for convergence for CG.  
In other words, the time reported in this table reflects how long it takes for CG to achieve 
the same accuracy as the proposed recursion method, which is why the CG relative residuals
are always less than those for the proposed recursion formula (except for one instance
where the CG method stagnated.)

\bigskip

\noindent \textbf{Analysis.}  
The three methods were run on numerous problems with various problems sizes, 
and we note that all methods achieve very small
relative residual errors for each of the problems we considered.  
Besides from memory issues, the direct method suffers from
significantly longer computational time, especially for the larger problems.  Generally,
the recursion algorithm takes about one-fourth the amount of time as the CG method.
The CG method requires $2M+2$ vector-vector products per iteration ($2M$ for the matrix-vector
product and $2$ for other vector-vector products) and in exact arithmetic
will converge in $2M+1$ iterations (because the matrix $B_M$ in (\ref{eqn-bfgs}) is the sum of
a scaled multiple of the identity with $2M$ rank-1 matrices, which means that $B_M$
has at most $2M+1$ unique eigenvalues).  However, from our computational experience, 
the number of iterations is closer to $4M$, which brings the total vector-vector 
multiplications for CG to around $8M^2$.  Meanwhile, the number of vector-vector
multiplications for the recursion formula in Algorithm \ref{alg-recursionC} is 
$(2M+1)(2M+2)/2 = 2M^2 + 3M + 1$, which explains why the CG algorithm takes
roughly four times as long to achieve the same accuracy as the the proposed recursion
algorithm.

\begin{table}
\centering
\begin{tabular}{|c|ccc|ccc|}
\hline
\multirow{2}{*}
{$n$}
& 
\multicolumn{3}{c|} 
{Relative Residual}\\
& Direct & CG & Recursion \\
\hline
$ \ \ \! 1,000$ & \texttt{\ 2.84e-14\ } & \texttt{\ 2.55e-14\ } & \texttt{\ 3.62e-14\ }
\\
$ \ \ \! 2,000$ & \texttt{\ 6.59e-14\ } & \texttt{\ 5.57e-14\ } & \texttt{\ 2.95e-13\ }
\\
$ \ \ \! 5,000$ & \texttt{\ 1.02e-13\ } & \texttt{\ 7.83e-14\ } & \texttt{\ 8.83e-14\ }
\\
$ 10,000$ & \texttt{\ 1.39e-13\ } & \texttt{\ 1.09e-13\ } & \texttt{\ 1.11e-13\ }
\\
$ 20,000$ & \texttt{\ 2.63e-14\ } & \texttt{\ 2.10e-14\ } & \texttt{\ 2.14e-14\ } 
\\
\hline
\end{tabular}
\caption{A sample run comparing the relative residuals of the solutions using 
the Matlab ``backslash" command, conjugate gradient method, and the proposed
recursion formula.   The relative residuals for the recursion formula are used as the criteria
for convergence for CG.  }
\label{table:smallresid}
\end{table}

\begin{table}
\centering
\begin{tabular}{|c|ccc|}
\hline
\multirow{2}{*}
{$n$}
& 
\multicolumn{3}{c|} 
{Time (sec)}\\
& Direct & CG & Recursion \\
\hline
$ \ \ \! 1,000$ & \texttt{\ \ 0.0311 \ } & \texttt{\ 0.0078 } & \texttt{\ \ 0.0015 \ }
\\
$ \ \ \! 2,000$ & \texttt{\ \ 0.2068 \ } & \texttt{\ 0.0099 } & \texttt{\ \ 0.0019 \ }
\\
$ \ \ \! 5,000$ & \texttt{\ \ 1.3692 \ } & \texttt{\ 0.0211 } & \texttt{\ \ 0.0048 \ }
\\
$ 10,000$ & \texttt{\ \ 8.0280 \ } & \texttt{\ 0.0306 } & \texttt{\ \ 0.0083 \ }
\\
$ 20,000$ & \texttt{\ 51.7772 \ } & \texttt{\ 0.0862 } & \texttt{\ \ 0.0160 \ }
\\
\hline
\end{tabular}
\caption{The computational times to achieve the results in Table \ref{table:smallresid}.}
\label{table:smalltime}
\end{table}

\begin{table}
\centering
\begin{tabular}{|c|cc|cc|}
\hline
\multirow{2}{*}
{$n$}
& 
\multicolumn{2}{c|}
{Relative residual} &
\multicolumn{2}{c|} 
{Time (sec)}\\
& CG & Recursion 
& CG & Recursion \\
\hline
$ \quad \ \ \! 100,000$ & \texttt{\ 8.71e-14\ } & \texttt{\  1.50e-13\ } & \texttt{ \ \ 0.2339 \ }
& \texttt{\ \ 0.0584 \ }
\\
$ \quad \ \ \! 200,000$ & \texttt{\ 1.47e-14\ } & \texttt{\ 3.27e-14\ } & \texttt{ \ \ 0.4686 \ }
& \texttt{\ \ 0.1155 \ }
\\
$ \quad \ \ \! 500,000$ & \texttt{\ \  4.90e-14$^{\star}$ } & \texttt{\ 3.55e-14\ } & \texttt{ \ \ 1.6996 \ }
& \texttt{\ \ 0.3587 \ }
\\
$ \ \ \! 1,000,000$ & \texttt{\ 9.99e-15\ } & \texttt{\ 1.03e-14\ } & \texttt{ \ \ 6.0068 \ }
& \texttt{\ \ 1.0914 \ }
\\
$ \ \ \! 2,000,000$ & \texttt{\ 1.10e-13\ } & \texttt{\ 6.54e-13\ } & \texttt{ \ 15.9798 \ }
& \texttt{\ \ 2.5653 \ }
\\
$ \ \ \! 5,000,000$ & \texttt{\ 3.08e-14\ } & \texttt{\ 4.84e-14\ } & \texttt{ \ 30.2865 \ } 
& \texttt{\ \ 6.2738 \ }
\\
$ 10,000,000$  & \texttt{\ 1.33e-14\ } & \texttt{\ 3.97e-14\ } & \texttt{ \ 67.3049 \ } 
& \texttt{\ 11.5946 \ }
\\ \hline
\end{tabular}
\caption{Comparison between the proposed recursion method and the built-in conjugate gradient
(CG) method in Matlab.  The relative residuals for the recursion formula are used as the criteria
for convergence for CG.  In other words, the time reported in this table reflects how long it takes for CG
to achieve the same accuracy as the proposed recursion method.  $^{\star}$In this case,
CG terminated without converging to the desired tolerance because ``the method stagnated."}
\label{table-large}
\end{table}

\section{Extensions}
\label{sec-extensions}
The proposed recursion algorithm also
computes products of the form $(B_k+D)^{-1}z$, where $D$ is any
positive-definite diagonal matrix.  In this case, we must assume that each
diagonal entry in $D$ satisfies $d_{ii} \ge \sigma$ for some $\sigma > 0$.
Provided $\gamma_k\sigma>\epsilon$, a theorem similar to Theorem
\ref{thrm-maa2} will hold true \textsl{mutatis mutandis}: the only
steps in the proof that need changing are (\ref{eqn-aC0a}), which becomes
$$
	a_0^TC_0^{-1}a_0 
	=
	\frac{\gamma_k^{-2}}{\gamma_k^{-1}s_0^Ts_0} s_0^T (B_0 + D)^{-1}s_0
	=
	\frac{1}{\gamma_k s_0^Ts_0} \sum_{i=1}^n \frac{1}{\gamma_k^{-1} + d_{ii}} (s_0)_i^2
	 \le 
	\frac{1}{1 + \gamma_k \sigma},
$$
and the cascading equations in (\ref{eqn-zCkz}), whose $\sigma \| z \|_2^2$
terms become $z^TDz$, which is greater than or equal to $\sigma \| z
\|_2^2$.  Additionally, the recursion formula for diagonal updates need not
be limited to L-BFGS systems.  In particular, it is also applicable to other
quasi-Newton systems where a recursion formula exists and the quasi-Newton
matrices are guaranteed to be positive definite.  For example, the proposed
recursion will work with quasi-Newton matrices using the DFP updating
formula, but it will not work with quasi-Newton matrices based on SR1, which
are not guaranteed to be positive definite (for more information on the DFP
and SR1 method see, for example,~\cite{NocW06}).

\section{Concluding remarks}
\label{sec:conc}

In this paper, we proposed an algorithm based on the SMW formula to solve
systems of the form $B+\sigma I$, where $B$ is an $n\times n$ L-BFGS
quasi-Newton matrix.  We showed that as long as $\gamma\sigma >\epsilon$,
the algorithm is well-defined.  The algorithm requires at most $M^2$ vector
inner products.  (Note: We assume that $M\ll n$, and thus, $M^2$ is also
significantly smaller than $n$.)
While the algorithm is designed to handle constant diagonal updates of a
quasi-Newton matrix, it can be extended to handle general diagonal updates
of a quasi-Newton matrix.  Furthermore, this algorithm can be extended to
handle any quasi-Newton updating that ensures the quasi-Newton matrices are
all positive definite.  
The algorithm proposed in this paper can be found at 
\texttt{http://www.wfu.edu/$\sim$erwayjb/software}.  

\newpage





\bibliographystyle{model1b-num-names}
\bibliography{jbe}







\end{document}

%% file: algmacros.tex
\newcommand{\Astrut}{\rule[-.5ex]{0pt}{2ex}}%
\newcommand{\struta}{\rule[-1.25ex]{0pt}{2ex}}%
\newcommand{\strutb}{\rule[-.9ex]{0pt}{2ex}}%
\newcommand{\mystrut}{\vrule height9.5pt depth1.5pt width0pt}

\newcommand{\agap}{\hspace{1.5ex}}
\newcommand{\AlgSec}[1]{{\bf[#1.]}\vrule height10pt depth1.5pt width0pt}

\newcommand{\tab}{\par\noindent\mystrut}
\newcommand{\tabb}{\tab\hskip1.5em}
\newcommand{\tabbb}{\tab\hskip3.5em}
\newcommand{\tabbbb}{\tab\hskip5.5em}
\newcommand{\tabbbbb}{\tab\hskip7.5em}
\newcommand{\tabbbbbb}{\tab\hskip9.5em}

\newcommand{\alg}[1]{\par\noindent\mystrut\ignorespaces\hbox to\textwidth{#1\hfill}}
\newcommand{\algt}[1]{\par\noindent\mystrut\hbox to\textwidth{\ignorespaces\hskip1.5em#1\hfill}}
\newcommand{\algtt}[1]{\par\noindent\mystrut\hbox to\textwidth{\ignorespaces\hskip3.5em#1\hfill}}
\newcommand{\algttt}[1]{\par\noindent\mystrut\hbox to\textwidth{\ignorespaces\hskip5.5em#1\hfill}}
\newcommand{\algtttt}[1]{\par\noindent\mystrut\hbox to\textwidth{\ignorespaces\hskip7.5em#1\hfill}}
\newcommand{\algttttt}[1]{\par\noindent\mystrut\hbox to\textwidth{\ignorespaces\hskip9.5em#1\hfill}}
\newcommand{\clg}[2]{\par\noindent\mystrut\hbox to\textwidth{\ignorespaces#1\hfill[#2]}}
\newcommand{\clgt}[2]{\par\noindent\mystrut\hbox to\textwidth{\ignorespaces\hskip1.5em#1\hfill[#2]}}
\newcommand{\clgtt}[2]{\par\noindent\mystrut\hbox to\textwidth{\ignorespaces\hskip3.5em#1\hfill[#2]}}
\newcommand{\clgttt}[2]{\par\noindent\mystrut\hbox to\textwidth{\ignorespaces\hskip5.5em#1\hfill[#2]}}
\newcommand{\clgtttt}[2]{\par\noindent\mystrut\hbox to\textwidth{\ignorespaces\hskip7.5em#1\hfill[#2]}}
\newcommand{\clgttttt}[2]{\par\noindent\mystrut\hbox to\textwidth{\ignorespaces\hskip9.5em#1\hfill[#2]}}

\newcommand{\tac}{\par\noindent\mystrut\hskip1em}
\newcommand{\tabc}{\tab\hskip2.5em}
\newcommand{\tabcc}{\tab\hskip4.5em}
\newcommand{\tabccc}{\tab\hskip6.5em}
\newcommand{\tabcccc}{\tab\hskip8.5em}
\newcommand{\tabccccc}{\tab\hskip10.5em}

\newcommand{\REPEAT}{\textbf{repeat}\Astrut}
\newcommand{\WHILE}{\textbf{while}\hskip2pt}
\newcommand{\FOR}{\textbf{for}\hskip2pt}
\newcommand{\IF}{\textbf{if}\hskip2pt}

\newcommand{\logical}[1]{\{\mbox{#1}\}}
\newcommand{\proc}[1]{\setbox4=\hbox{\noindent\strut#1}\boxit{\box4}}
\newcommand{\procc}[1]{\setbox4=\vbox{\hsize 20pc \noindent\strut#1}\boxit{\box4}}
\newcommand{\DO}{\hskip2pt\textbf{do}\hskip2pt\Astrut}
\newcommand{\UNTIL}{\textbf{until}\hskip2pt}
\newcommand{\BEGIN}{\textbf{begin}}
\newcommand{\END}{\textbf{end}}
\newcommand{\ENDIF}{\textbf{end if}}
\newcommand{\ENDWHILE}{\textbf{end while}}
\newcommand{\ENDREPEAT}{\textbf{end repeat}}
\newcommand{\ENDDO}{\textbf{end do}}
\newcommand{\ENDFOR}{\textbf{end (for)}}
\newcommand{\THEN}{\hskip2pt\textbf{then}\hskip2pt}
\newcommand{\ELSE}{\hskip2pt\textbf{else}\hskip2pt}
\newcommand{\FIRSTELSE}{\textbf{else}\hskip2pt}
\newcommand{\ELSEIF}{\textbf{else if}}
\newcommand{\RETURN}{\textbf{return}}
\newcommand{\FIRST}{\hskip-2pt}  

\newcommand{\STOP}{\mathbf{stop}}
\newcommand{\BREAK}{\mathbf{break}}
\newcommand{\TRUE}{\mathbf{true}}
\newcommand{\FALSE}{\mathbf{false}}
\newcommand{\true}{\mathbf{true}}
\newcommand{\false}{\mathbf{false}}
\newcommand{\NOT}{\mathop{\mathbf{not}\,}}
\newcommand{\AND}{\mathop{\;\mathbf{and}\;}}
\newcommand{\OR}{\mathop{\;\mathbf{or}\;}}


\newcommand{\badSteps}{\mathit{badSteps}}
\newcommand{\converged}{\mathit{converged}}
\newcommand{\complete}{\hbox{\it complete}}
\newcommand{\CSpoint}{\hbox{\it subspace\_stationary\_pt}}
\newcommand{\exchange}{\mbox{\it exchange}}
\newcommand{\exit}{\hbox{\it exit}}
\newcommand{\found}{\hbox{\it found}}
\newcommand{\feasible}{\mathit{feasible}}
\newcommand{\hitcon}{\mbox{\it hit\_constraint}}
\newcommand{\addcon}{\mbox{\it add\_constraint}}
\newcommand{\Hsingular}{\mbox{\it singular\_H}}
\newcommand{\imprvd}{\hbox{\it improved}}
\newcommand{\indef}{\hbox{\it indefinite}}
\newcommand{\minimizer}{\mbox{\it subspace_minimizer}}
\newcommand{\optimal}{\mathit{optimal}}
\newcommand{\posdef}{\hbox{\it positive\_definite}}
\newcommand{\unbndd}{\hbox{\it unbounded}}
\newcommand{\semidef}{\hbox{\it positive\_semidefinite}}
\newcommand{\stationary}{\hbox{\it stationary\_point}}
\newcommand{\subspace}{\hbox{\it subspace\_convergence}}
\newcommand{\singular}{\mathit{singular}}
\newcommand{\tol}{\mathit{tol}}
\newcommand{\uncon}{\mbox{\it unconstrained\_step}}
\newcommand{\vertex}{\mbox{\it vertex}}
\newcommand{\wschosen}{\hbox{\it working\_set\_chosen}}
\newcommand{\Wsingular}{\mbox{\it singular\_W}}
\newcommand{\Asingular}{\mbox{\it singular\_A}}
